\documentclass[12pt]{article}   
\usepackage{graphicx,psfrag}
\usepackage{color}
\usepackage{array}
\usepackage{colortbl}
\usepackage{tabularx}
\newtheorem{theorem}{Theorem} [section]
\newtheorem{corollary}[theorem]{Corollary}

\newtheorem{lemma}[theorem]{Lemma}

\newenvironment {proof} {{\it
Proof.}}{\hspace*{\fill}$\Box$\par\vspace{4mm}}
\usepackage{amsmath}
\usepackage{amsfonts}
\newfont{\bb}{msbm10}

\def\:{\! :\!}

\newcommand\0{{\bf 0}}

\usepackage{tikz}

\begin{document}

\title{Patterns of Alternating Sign Matrices}

 \author{Richard A. Brualdi, Kathleen P. Kiernan, \\Seth A. Meyer, Michael W. Schroeder\\
 Department of Mathematics\\
 University of Wisconsin\\
 Madison, WI 53706\\
 {\tt \{brualdi,kiernan,smeyer,schroede\}@math.wisc.edu}
 }

\maketitle

\centerline{\it In admiration, to Avi Berman, Moshe Goldberg, and Raphi Loewy}

 \begin{abstract} 
We initiate a study of the zero-nonzero patterns of $n\times n$ alternating sign matrices. We characterize the row (column) sum vectors of these patterns and determine their minimum term rank. In the case of connected alternating sign matrices, we find the minimum number of nonzero entries and characterize the case of equality. We also study symmetric alternating sign matrices, in particular, those with only zeros on the main diagonal. These give rise to alternating signed graphs without loops, and we determine the maximum number of edges in such graphs. We also consider $n\times n$ alternating sign matrices whose patterns are maximal within the class of all $n\times n$ alternating sign matrices.

\medskip
\noindent {\bf Key words and phrases: (symmetric) alternating sign matrix, ASM,  pattern, row (column) sum vector, alternating signed graph, term rank. } 

\noindent {\bf Mathematics  Subject Classifications: 05B20, 05C22, 05C50, 15B36} 
\end{abstract}

\section{Introduction}
An {\it alternating sign matrix}, henceforth abbreviated ASM, is an $n\times n$ $(0,+1,-1)$-matrix such that the $+1$s and $-1$s alternate in each row and column, beginning and ending with a $+1$. For example,
\[\left[\begin{array}{rrrrr}
0&0&+1&0&0\\
0&+1&-1&+1&0\\
+1&-1&+1&-1&+1\\
0&+1&-1&+1&0\\
0&0&+1&0&0\end{array}\right]\mbox{ and }
\left[\begin{array}{rrrrr}
0&0&0&+1&0\\
+1&0&0&-1&+1\\
0&+1&0&0&0\\
0&0&0&+1&0\\
0&0&+1&0&0\end{array}\right]\]
are ASMs.
Since every permutation matrix is an ASM, ASMs can be thought of as generalizations of permutation matrices. 
The following elementary properties  of  ASMs follow from their definition:
\begin{enumerate}
\item[\rm (i)] The number of nonzeros in each row and column
is odd.
\item[\rm (ii)] The first and last nonzero in each row and column is a $+1$.
\item[\rm (iii)] The first and last rows and columns contain exactly one $+1$ and no $-1$s.
\item[\rm (iv)] The partial row  sums of the entries in each row, starting from the first (or last) entry, equal 0 or 1. A similar property hold for columns. Moreover, all row and column sums equal 1.
\end{enumerate}
Reversing the order of the rows or columns of an ASM results in another ASM. More generally, thinking of an $n\times n$ matrix as an $n\times n$ square, applying  the action of an element of the dihedral group of order $8$ to an ASM results in another ASM. 
A  matrix obtained from an ASM  by row and column permutations  is generally not an ASM, even in the case of a simultaneous row and column permutation. In particular, being an ASM is not a combinatorial property of a matrix.

The substantial interest in ASMs in the mathematics community originated from the alternating sign matrix conjecture of Mills, Robbins,  and Rumsey  \cite{MRR} in 1983; see \cite{DB,JP,DPR} for a history of this conjecture and its relation with other combinatorial constructs.  This conjecture, proved by Zeilberger \cite{DZ} and also by Kuperberg \cite{GK}, asserts that the number of $n\times n$ ASMs equals
\[\frac{1!4!7!\cdots (3n-2)!}{n!(n+1)!(n+2)!\cdots (2n-1)!}.\]
In this paper we are concerned with the zero-nonzero pattern of an ASM.

An $n\times n$ ASM $A$ has a unique decomposition of the form
\[A=A_1-A_{2}\]
where $A_1$ and $A_{2}$ are $(0,+1)$-matrices without any $+1$ in the same position.
For example,
\begin{equation}\label{1eq:pattern}
\left[\begin{array}{rrrr}
0&+1&0&0\\
+1&-1&+1&0\\
0&0&0&+1\\
0&+1&0&0\end{array}\right]=
\left[\begin{array}{rrrr}
0&+1&0&0\\
+1&0&+1&0\\
0&0&0&+1\\
0&+1&0&0\end{array}\right]-
\left[\begin{array}{rrrr}
0&0&0&0\\0&+1&0&0\\0&0&0&0\\0&0&0&0\end{array}\right].\end{equation}
The {\it pattern} of an ASM $A$ is the $(0,+1)$-matrix
$\widetilde{A}=A_1+A_2$. The ASM in (\ref{1eq:pattern}) has pattern 
\[\left[\begin{array}{rrrr}
0&+1&0&0\\
+1&+1&+1&0\\
0&0&0&+1\\
0&+1&0&0\end{array}\right],\]
which, dropping the symbol $1$ or  dropping both symbols $1$ and $0$, we also write as
%
\[
\left[\begin{array}{rrrr}
0&+&0&0\\
+&+&+&0\\
0&0&0&+\\
0&+&0&0\end{array}\right]\mbox{ or }
\left[\begin{array}{r|r|r|r}
&+&&\\ \hline
+&+&+&\\ \hline
&&&+\\ \hline
&+&&\end{array}\right]
.\]

An $n\times n$  ASM $A$ can be viewed as a signed biadjacency matrix of a bipartite graph $G\subseteq K_{n,n}$, equivalently,  the biadjacency matrix of a signed bipartite graph,  where  the vertices in each part of the bipartition are linearly  ordered and where the edges are signed $\pm 1$.  Similarly,  $n\times n$ symmetric ASMs   can be viewed as  the signed adjacency matrix of a loopy\footnote{Loopy refers to the fact that loops are allowed but not required.} graph $G\subseteq K_n$, equivalently, the adjacency matrix of a signed loopy graph, where the vertices are   linearly ordered and  where the edges are labeled $\pm 1$.
 As appropriate, we  use the terminology {\it alternating signed bipartite graph}, abbreviated ASBG,  and {\it alternating signed loopy graph}, abbreviated ASLG. If there are only 0s on the main diagonal of a symmetric ASM, that is, there are no loops in the graph, then  we have an {\it alternating signed graph}, abbreviated ASG. 

Our main goal in this paper is to investigate possible patterns of ASMs
and to determine their properties.
In Section 2 we discuss some basic properties of patterns of ASMs.
 The maximum number of nonzeros in an $n\times n$ ASM (an ASBG) is easily determined along with the case of equality.
 The minimum number of nonzero entries in an $n\times n$ ASM is $n$, since permutation matrices are ASMs.
Let $A$ be an $n\times n$ ASM, and let $R=(r_1,r_2,\ldots,r_n)$ and $S=(s_1,s_2,\ldots,s_n)$ be the row sum vector and column sum vector of the pattern of $A$. Then the components of $R$ and $S$ 
are odd positive integers with $r_1=s_1=r_n=s_n=1$ and
\[r_1+r_2+\cdots+r_n=s_1+s_2+\cdots+s_n.\]
We  characterize the possible row sum (column sum) vectors of patterns of ASMs.  In Section 3 we determine the minimum number of nonzeros in an ASM whose associated signed bipartite graph is connected.
 In Section 4 we consider ASMs which are maximal in the sense that  it is not possible to change one or more  $0$ to  $\pm1$ resulting in another ASM.
 Recall that the {\it term rank} $\rho(X)$ of a matrix $X$ is the maximum  possible number of nonzeros in $X$ with no two from the same row and column, and that by the K\"onig-Eg\'ervary theorem, $\rho(X)$ is the minimum number of rows and columns of $X$ that contain all the nonzeros of $X$. The term rank of an $n\times n$ ASM may be $n$, as the permutation matrices show. In  Section 5  we determine the smallest possible term rank of an $n\times n$ ASM.
  In Section 6 we determine the maximum  number of edges in an ASG of order $n$, that is, the maximum  number of nonzeros in  an $n\times n$ symmetric ASM with only $0$s on the main diagonal.  We also  determine when equality occurs.
  Some concluding remarks are given in Section 7.

\section{Patterns of ASMs}

We first observe that there is a simple construction which  shows that  given any $k\times l$ $(0,+1,-1)$-matrix $B$, there is an ASM that contains $B$ as a  submatrix.  To construct such an ASM, we proceed  as follows:

\begin{enumerate}
\item[\rm (a)]
  If some row of $B$ is a zero row, we put in  a new column anywhere which has a $+1$ in that row and $0$s elsewhere.
 \item[\rm (b)] If in $B$ the first nonzero entry of some row is a $-1$, we put in a new column anywhere to its left whose only nonzero entry is  a $+1$ in that row.
 \item[\rm (c)] If the last nonzero entry of some row of $B$  is a $-1$, we put in a new column anywhere to its right whose only nonzero entry is a $+1$ in that row. 
 \item[\rm (d)] If $B$ has two $-1$s  in row $i$, with no nonzeros in between, we put in a new column anywhere in between the two $-1$s whose only nonzero entry is a $+1$ in that row $i$. 
 \item[\rm (e)] If $B$ has two $+1$s in row $i$ with no nonzeros in between, we put in a  new column anywhere between the two $+1$s and two new rows, one anywhere above row $i$ and one anywhere below  row $i$. The new column has a $-1$  in row $i$ and a $+1$ in the two new rows.
 \item[\rm (f)]
 We repeat (a)-(e)  using   the  columns of $B$ in place of the rows.
 \end{enumerate}
Applying this construction, we obtain  an ASM $A$ containing $B$ as a submatrix; we say that such an $A$ is obtained from $B$ by an {\it elementary ASM expansion} of $B$. In general, an ASM containing $B$ as a submatrix is an {\it ASM expansion} of $B$. It follows from our construction that there can be no forbidden submatrix condition that can be used to characterize ASMs.

Now let $A=[a_{ij}]$ be an $n\times n$ ASM
with pattern $\widetilde{A}$. 
Let $R=(r_1,r_2,\ldots,r_n)$ and $S=(s_1,s_2,\ldots,s_n)$ be the row sum vector and column sum vector of $\widetilde{A}$. 

Let ${\bold  j}_n$ equal the $n$-vector $(1,1,\ldots,1)$ of all 1s, and let
${\bold  k}_n$ be the $n$-vector $(1,3,5,\ldots,5,3,1)$ which reads backward the same as it reads forward. Thus
${\bold k}_6=(1,3,5,5,3,1)$ and ${\bold k}_7=(1,3,5,7,5,3,1)$.
There is a well-known special $n\times n$ ASM $D_n$ which we call the {\it diamond ASM} and  illustrate in (\ref{eq:D_n}) for $n=6$ and $7$, it being obvious how to generalize to arbitrary $n$:
\begin{equation}\label{eq:D_n}
D_6=\left[\begin{array}{c|c|c|c|c|c}
&&&+&&\\ \hline
&&+&-&+&\\ \hline
&+&-&+&-&+\\ \hline
+&-&+&-&+&\\ \hline
&+&-&+&&\\ \hline
&&+&&&\end{array}\right],\quad
D_7=\left[\begin{array}{c|c|c|c|c|c|c}
&&&+&&&\\ \hline
&&+&-&+&&\\ \hline
&+&-&+&-&+&\\ \hline
+&-&+&-&+&-&+\\ \hline
&+&-&+&-&+&\\ \hline
&&+&-&+&&\\ \hline
&&&+&&&\end{array}\right].
\end{equation}
For $n$ even, the matrix obtained from $D_n$ by taking its rows in the reverse order  is also an ASM and we also refer to it as a diamond ASM. The row and column sum vectors of $D_n$ equal ${\bf k}_n$.

\begin{lemma}\label{lem:pat1}
If $A$ is an $n\times n$ ASM whose pattern $\widetilde{A}$ has row sum vector $R$ and column sum vector $S$, then
\[{\bold j}_n\le R,S \le {\bold k}_n\;
 (\mbox{entrywise}).\]
 Moreover,  $R=S={\bold j}_n$ if and only if  $A$ is a permutation matrix, and $R=S={\bold k}_n$ if and only if $A$ is a diamond ASM.
\end{lemma}

\begin{proof}
The assertions involving  ${\bold j}_n$ are obvious since an ASM must have at least one nonzero in each row and column, and permutation matrices are ASMs. The assertions involving ${\bold k}_n$ are also straightforward. The first and last row (respectively, column) of an ASM can contain only one $+1$ and no $-1$s. Let $i>1$.
The $-1$s in row $i$ of an ASM can only be in those columns for which the corresponding column sum of the leading $(i-1)\times n$ submatrix $A_{i-1}$ of $A$ equals 1.
 Since the row sums of $A_{i-1}$ equal $1$ and the column sums
 of $A_{i-1}$ are nonnegative, exactly $i-1$ 
 of the column sums of  $A_{i-1}$  are equal to $1$.  
 We conclude that there are at most $i$\;  $+1$s and at most $(i-1)$\; $-1$s in row $i$, and hence at most  $2i-1$ nonzeros in row $i$. The same argument applies to the column sums taken from bottom to top. This proves that $R\le {\bold k}_n$, and in a similar way we get that $S\le {\bold k}_n$.
 
 If $n$ is odd and $R=S={\bf k}_n$, then row $(n+1)/2$ has only nonzero entries,  alternating between $+1$ and $-1$, and the ASM is uniquely determined as $D_n$. If $n$ is even, then rows $(n/2)\pm 1$ contain exactly one $0$.  Because $R={\bold k}_n$, this unique $0$ must be in the first column of one of these rows and the last column of the other, and we again get that $A$ is $D_n$.
\end{proof}

The linear constraints  ${\bold j}_n\le R,S \le {\bold k}_n$, $r_1=r_n=s_1=s_n=1$,  and $r_1+r_2+\cdots+r_n=s_1+s_2+\cdots+s_n$ for vectors $R=(r_1,r_2,\ldots,r_n)$ and $S=(s_1,s_2,\ldots,s_n)$  of positive odd integers  are not sufficient to guarantee the existence of an $n\times n$ ASM whose pattern has row sum vector $R$ and column sum vector $S$. For example, 
if $n=5$ and $R=S=(1,1,5,1,1)$, there is no $5\times 5$ ASM $A$ whose pattern  has this row and column sum vector. This is because the third row of such an $A$ 
must be $[+\  -\ +\  -\ +]$ and the third column must be the transpose of this vector. Hence the row sum and column sum vector of such an $A$ must be ${\bold k}_5$ and not $(1,1,5,1,1)$. This leads to another necessary condition for integral $n$-vectors $R=(r_1,r_2,\ldots,r_n)$ and $S=(s_1,s_2,\ldots,s_n)$, namely,
\[\frac{r_i-1}{2}\le  \left|\{j:s_j\ge 3\}\right|\quad (1\le i\le n).\]
This is because row $i$ of an ASM contains $(r_i-1)/2$\;  $-1$s and so there are at least this many columns with at least 3 nonzeros. This kind of argument can be carried further. Consider consecutive rows $i$ and $i+1$ of an ASM. Their $-1$s occur in different columns. Thus there are at least 
\[\frac{r_i-1}{2}+\frac{r_{i+1}-1}{2}\]
columns with a $-1$ and thus at least this many columns with at least three nonzeros.
 Therefore
\[\frac{r_i-1}{2}+\frac{r_{i+1}-1}{2}\le \left|\{j:s_j\ge 3\}\right|\quad (1\le i\le n).\]

There do not seem to be any easily formulated necessary and sufficient conditions for $R$ and $S$ to be the row and column sum vectors of an ASM. However,
the inequalities ${\bold j}_n\le R \le {\bold k}_n$ for  an $n$-vector $R$ of  positive odd integers are  sufficient to guarantee the existence of an $n\times n$ ASM whose pattern has row sum vector $R$.

\begin{theorem}\label{th:pat1}
Let $R=(r_1,r_2,\ldots,r_n)$ be a vector of positive odd integers.
 Then there is an $n\times n$ ASM whose pattern has row sum vector $R$ if and only if  
 \begin{equation}\label{eq:suff} {\bold j}_n\le R \le {\bold k}_n.\end{equation}
\end{theorem}
 
 \begin{proof} We note that (\ref{eq:suff}) implies that $r_1=r_n=1$.
 By Lemma \ref{lem:pat1}, (\ref{eq:suff}) is satisfied by  the row sum vector of the pattern of an $n\times n$ ASM. Now suppose that (\ref{eq:suff}) holds. We show how to construct an ASM whose pattern has row sum vector $R$. The ASM  is obtained by combining the results of two  smaller constructions which we now describe. We first assume that $n$ is an odd integer, $n=2m+1$.

 \smallskip\noindent
 Construction, Part I:  For $1\le k\le m+1$ we inductively construct a $k\times (2k-1)$ $(0,+1,-1)$-matrix $A_k$ such that the rows satisfy the alternating $\pm1$ condition of ASMs with the row sums of its pattern equal to $(r_1,r_2,\ldots,r_k)$, and the columns also satisfy the alternating $\pm 1$ condition from top to bottom except for the fact that the full column sum vector is $(1,0,1,0,\ldots,1,0,1)$ (and not $(1,1,\ldots,1)$ as it would be for an ASM). If $k=1$, then  $A_1=[+ ]$. If $k=2$, then 
 \[A_2=\left[\begin{array}{rrr}
 0&0&+\\
 +&0&0\end{array}\right]\mbox{ if $r_2=1$, and }
 A_2=\left[\begin{array}{rrr}
 0&+&0\\ +&-&+\end{array}\right]\mbox{ if $r_2=3$.}\]
Suppose that $k\ge 2$ and we have constructed $A_{k}$ to satisfy our requirements. If $r_{k+1}=2k+1$, then we border $A_k$ with a zero column on the left and right, and adjoin the $(2k+1)$-vector
$[+\  0\  +\  0\ \ldots\  +\ 0\ +]$ as a new row. The resulting $(k+1)\times (2k+1)$-matrix $A_{k+1}$ satisfies our requirements. Now suppose that $r_{k+1}<2k+1$. Then  we adjoin two zero columns to $A_k$, a zero column on the left and a zero column between the $r_{k+1}$ and $r_{k+1}+1$ columns of $A_k$; we then adjoin as a new row the $(2k+1)$-vector  
$[+\ -\ +\ -\ \ldots\  +\ -\ +\ 0\ \ldots\ 0]$.
The resulting $(k+1)\times (2k+1)$ matrix $A_k$ satisfies our requirements.

In this part of the  construction, we finish with an $(m+1)\times (2m+1)$ $(0,+1,-1)$-matrix $A_{m+1}$ satisfying our requirements where the row sum vector of the pattern of $A_{m+1}$ is $(r_1,r_2,\ldots,r_{m+1})$. Since the column sum vector of $A_{m+1}$ is $(1,0,1,0,\ldots,1,0,1)$, the last nonzero entry, if there is a nonzero entry,  in its odd numbered columns is a $1$, and in the even numbered columns it is a $-1$.

 \smallskip\noindent
 Construction,  Part II: Using Construction I  we can obtain an $m\times (2m-1)$ $(0,+1,-1)$-matrix $A'_m$,
such that the rows of $A'_m$ satisfy the alternating $\pm1$ condition of ASMs with the row sums of its pattern equal to $(r_{2m+1},r_{2m},\ldots,r_{m+2})$, and the columns also satisfy the alternating $\pm 1$ condition  except for the fact that the full column sums are $(1,0,1,0,\ldots,1,0,1)$. Let $A''_m$ be the matrix obtained by bordering $A'_m$ by a zero column on the left and right, and then taking its rows in the reverse order. The columns sums of $A''_m$ are $(0,1,0,1,\ldots,0,1,0)$. The first nonzero entry of the odd numbered columns of $A''_m$, if there is one,  is a $-1$ while the first nonzero entry in the even numbered columns of $A''_m$ is a $+1$. It follows that the matrix
\[\left[\begin{array}{cc} 
 A_{m+1}\\ A''_m\end{array}\right]\]
 is a $(2m+1)\times  (2m+1)$  ASM whose pattern has row sum vector
 $R$.
 
 Now assume that $n$ is an even integer $n=2m$. Then using the ideas above, we obtain an $m\times (2m-1)$ $(0,+1,-1)$-matrix whose rows satisfy the alternating $\pm 1$ condition of ASMs with the row sums of its pattern equal to $(r_1,r_2,\ldots,r_m)$ and whose columns satisfy the alternating $\pm 1$ condition except for the fact that the full column sums equal $(1,0,1,0,\ldots,1)$. We then append a column of all $0$s to obtain an $m\times 2m$ matrix $A_m$ with full column sums  equal
 to  $(1,0,1,0,\ldots,1,0)$. We also obtain an $m\times 2m$ $(0,+1,-1)$-matrix $A_m'$ whose rows satisfy the alternating $\pm 1$ condition of ASMs with the row sums of its pattern equal to $(r_{2m},r_{2m-1},\ldots,r_{m+1})$ and whose columns satisfy the alternating $\pm 1$ condition except for the fact that the full column sums equal $(0,1,0,1,\ldots,0,1)$. Taking the rows of $A_m'$ in the reverse order
to produce the matrix $A_m''$, we obtain the required $2m\times 2m$
ASM
\[\left[\begin{array}{cc}
A_m\\ A_m''\end{array}\right].\]
  \end{proof}

\section{Minimal Connected ASMs}

Let $A=[a_{ij}]$ be an $n\times n$ ASM, and let $BG(A)\subseteq K_{n,n}$ be the alternating signed bipartite graph (ASBG) determined by $A$.
 If $BG(A)$ is  connected, then we call $A$ a {\it connected ASM}; otherwise, $A$ is a {\it disconnected ASM}. It follows easily that there exist permutation matrices $P$ and $Q$ such that
\begin{equation}\label{eq:conn}
PAQ=A_1\oplus A_2\oplus\cdots\oplus A_h\quad (h\ge 1)\end{equation}
where $A_1,A_2,\ldots,A_h$ are also ASMs and $BG(A_i)$ is connected for all $i$. If, for instance, $A$ is a permutation matrix, then in (\ref{eq:conn}) we get $h=n$ and $PAQ=I_n$.  In this section we show that the minimum number of nonzero entries of an $n\times n$ connected ASM equals $2n-1$ (thus $BG(A)$ is a tree) if $n$ is odd, and equals
$2n$ (thus $BG(A)$ is a unicyclic graph whose unique cycle has even length) if $n$ is even. In each case, we give a recursive  construction to obtain  all graphs attaining equality. If $A$ is an $n\times n$ ASM, then the degrees of all the vertices of $BG(A)$ are odd; if $BG(A)$ is a tree, then the sum of the degrees of the vertices in one part of the bipartition  equals  $2n-1$, the number  of edges of $BG(A)$. 
Hence $BG(A)$ can be a tree only if $n$ is odd.

Recall that the diamond ASM $D_3$ equals
\[\left[\begin{array}{c|c|c}
&+&\\ \hline
+&-&+\\ \hline
&+&\end{array}\right].\]
The signed bipartite graph $BG(D_3)$ is  shown in Figure 1.

\smallskip
\begin{center}
\begin{tikzpicture}
 \filldraw    (0,0) circle (3pt)
                  (1,0) circle (3pt) 
               (2,1) circle (3pt)
               (2,-1) circle( 3pt)
               (-1,1) circle (3pt)
               (-1,-1) circle (3pt);
             
\draw[thick]                   (1,0) -- (2,1) node[above,left=10pt]{$+$};
\draw[thick]                    (1,0)-- (2,-1) node[below,left=10pt]{$+$};
           \draw[thick]         (0,0)--(-1,1) node[above,right=10pt]{$+$};
\draw[thick]                     (0,0)--(-1,-1) node[below,right=10pt]{$+$};
\draw[thick] (0,0)--(1,0) node[left=12pt,above]{$-$};

\draw[xshift=1cm,yshift=-2.5cm] node {\bf Figure 1: $BG(D_3)$};
                    
\end{tikzpicture}
\end{center}
 
Let $A$ be an $n\times n$ ASM. Then we use the notation
$A\ast D_3$ to denote an $(n+2)\times (n+2)$ matrix obtained from $A$ by
identifying a $+$ of $A$ with a $+$ of $D_3$ and inserting two new rows and two new columns in order to embed $D_3$   in the resulting matrix. Neither the two new rows nor two new columns need be consecutive, but they must retain the same relative order as in $D_3$. 
The matrix $A\ast D_3$ is also an ASM, and we say that $A\ast D_3$ results from $A$ by {\it attaching} $D_3$.
The bipartite graph $BG(A\ast D_3)$  is obtained by identifying a positive edge of $BG(D_3)$ with a positive edge of $BG(A)$.
 For instance, if
\[A=\left[\begin{array}{c|c|c|c|c}
&&+&&\\ \hline
&+&-&&+\\ \hline
+&-&\cellcolor[gray]{0.8} +&&\\ \hline
&+&-&+&\\ \hline
&&+&&\end{array}\right],\]
then one possible $A\ast D_3$ is the matrix
\[\left[\begin{array}{c|c|c|c|c|c|c}
&&+&&&& \\ \hline 
&+&-&&&+& \\ \hline 
&&&\cellcolor[gray]{0.8} +&&& \\ \hline 
+&-&\cellcolor[gray]{0.8} +&\cellcolor[gray]{0.8} -
&&&\cellcolor[gray]{0.8} + \\ \hline 
&+&-&&+&& \\ \hline 
&&&\cellcolor[gray]{0.8} +&&& \\ \hline 
&&+&&&& \end{array}\right],\]
where the $+$ of $A$ in its position $(3,3)$ is identified with the $+$ of $D_3$ in its position $(2,1)$, and the two new rows are rows 3 and 6 and the two new columns  are columns 4 and 7. Deleting rows 3 and 6 and columns 4 and 7  of $A\ast D_3$ results in the original ASM $A$. 

Since $BG(D_3)$ is a tree, then by
starting with $I_1$ and attaching $D_3$s  we can construct for every odd integer $n$ an $n\times n$ connected ASM $A$ with $BG(A)$ a tree. Now suppose that $n$ is an even integer with $n\ge 4$. Then there are $n\times n$ connected ASMs with exactly $2n$ nonzero entries. One such basic family of ASMs is obtained as follows:  Let $m\ge 2$ be an integer, and  consider a cycle $C$ of even length $2m$ whose edges are (arbitrarily) either positive or negative. Depending on the signs of the edges at a vertex of $C$, we attach new edges to $C$ as follows: (i) if two positive edges meet at a vertex of $C$  we attach a negative edge to a new vertex and then two positive edges to that new vertex, (ii) if one positive and one negative edge meet at a vertex we attach a positive edge to a new vertex, (iii) if two negative edges meet at a vertex we attach three positive edges to new vertices. Since the number of edges attached to each vertex of $C$  is odd and $C$ has an even number of vertices, the total number of vertices of the resulting graph is even. If, for instance, the edges of $C$ alternate in sign, then we attach a positive (pendent) edge to each vertex of $C$. It is easy to see that   the resulting graph $G\subseteq K_{2m,2m}$ can always be realized (in many ways) as a $2m\times 2m$ ASM with exactly $2\cdot 2m$ nonzero entries; in fact, this construction is just a graphical description of the elementary ASM expansion construction discussed at the beginning of Section 2. For example, with $m=3$ and a cycle of length 6 whose edges alternate in sign,  we have
\[\left[\begin{array}{c|c|c|c|c|c}
&&&+&&\\ \hline
&&\cellcolor[gray]{0.8}  +&\cellcolor[gray]{0.8} -&&+\\ \hline
&+&&&&\\ \hline
+&\cellcolor[gray]{0.8} -&&\cellcolor[gray]{0.8} +&&\\ \hline
&\cellcolor[gray]{0.8} +&\cellcolor[gray]{0.8} -&&+&\\ \hline
&&+&&&\end{array}\right].\]
With $m=3$ and a cycle of length 6 with edges $+,+,-,+,-,-$, we have the $8\times 8$ ASM
\[\left[ \begin{array}{c|c|c|c|c|c|c|c}
&&+&&&&&\\ \hline
&\cellcolor[gray]{0.8} +&-&&\cellcolor[gray]{0.8} +&&&\\ \hline
&&+&&&&&\\ \hline
&&&+&\cellcolor[gray]{0.8} -&&\cellcolor[gray]{0.8} +&\\ \hline
+&\cellcolor[gray]{0.8} -&&&&+&\cellcolor[gray]{0.8} -&+\\ \hline
&&&&&&+&\\ \hline
&&&&+&&&\\ \hline
&+&&&&&&\end{array}\right].\]
We call an ASM constructed in this way a {\it basic unicyclic ASM}.

Given an  $n\times n$  connected ASM $A$ with $2n$ nonzero entries, then each matrix of the form $A\ast D_3$ is an $(n+2)\times (n+2)$ connected ASM with $2(n+2)$ nonzero entries.

\begin{theorem}\label{th:tree}
The minimum number of nonzero entries in an $n\times n$ connected ASM equals $2n-1$ if $n$ is odd and $2n$ if $n$ is even. If $n$ is odd, then an $n\times n$ connected ASM $A$ has exactly $2n-1$ nonzero entries if and only if $A$ equals  the $1\times 1$ identity matrix $I_1$ or $A$ can be obtained from  $I_1$ by recursively attaching $D_3$s. If $n$ is even, then  an $n\times n$ connected ASM $A$ has exactly $2n$ nonzero entries if and only if $A$ is either a basic unicyclic ASM or can be obtained from a basic unicyclic ASM  by recursively attaching $D_3$s.
\end{theorem}

\begin{proof}
It follows from  our preceding discussion that the minima are as given in the theorem, and the ASMs as described in the theorem attain the minima. We need to show that an ASM satisfying one of the minima can be obtained as described in the theorem.

First consider the case of $n$ odd. We may assume that $n\ge 3$. Let $A$ be an $n\times n$ ASM where $BG(A)$ is a tree $T$. Let $u$ be any pendent vertex of $T$ and root $T$ at $u$; the unique edge at $u$ is positive. Let $d$ be the largest distance of a vertex to $u$.  The case $d=1$ is trivial, and the case $d=2$ is not possible. Hence we may assume that $d\ge 3$.  The vertices of $T$ are partitioned into sets $V_0,V_1,\ldots,V_d$ where $V_i$ consists of all vertices at distance $i$ to $u$ $(0\le i\le d)$, and the only edges of $T$ join a vertex in $V_i$ to a vertex in $V_{i+1}$ for some $i$ with  $0\le i\le d-1$. For each $i$ with $1\le i\le d$ and each $v\in V_i$, there is a unique vertex $w\in V_{i-1}$ which is joined by an edge to $v$. 

Let $p$ be any vertex in $V_d$. Then $p$ is a pendent vertex which is joined by a (positive) edge to a unique vertex $q\in V_{d-1}$. 
The vertex $q$ must have  degree 3, and hence there exist vertices $r\in V_{d-2}$ and  $s\in V_d$ with a negative edge joining $q$ to $r$ and a positive edge joining $q$ to $s$ . The vertex $r$ must be incident with at least two positive edges, and so there exists another vertex $t\in V_{d-1}$ such that there is a positive edge joining $r$ and $t$. Since this edge is positive and $t\in V_{d-1}$, $t$ must be a pendent vertex. It follows that if $A'$ is the $(n-2)\times (n-2)$ submatrix of $A$ obtained by deleting the rows or columns corresponding to vertices $p,q,s,t$, then  $A$ is obtained from $A'$ by attaching a $D_3$. We need to know that we can choose $t$ and $q$ so that $A'$ is also an ASM, that is, so that the row or column corresponding to vertex $r$ is alternating in sign; there may be many positive and negative edges at vertex $r$. Since there is only one edge going from $r$ to a vertex in $V_{d-3}$ (one sign in the row or column corresponding to $r$), there must be such a consecutive pair.

Now assume that $n\ge 4$ is even, and that $A$ is an $n\times n$ connected ASM with $2n$ nonzero entries. Then $BG(A)$ contains a cycle of even length, and hence a basic unicyclic ASM $U$. If $A=U$ we are done. Otherwise, there exists a tree rooted at at least one of the vertices of $BG(U)$.  Now, using an argument similar to that used in the odd case, we complete the proof.
\end{proof}

 \section{Term Rank of ASMs}

The {\it term rank} $\rho(B)$ of  a matrix $B$ is the maximum number  of nonzeros of $B$ with no two from the same row or column. The term rank of $B$ also equals the smallest number of rows and columns that contain all the nonzeros of $B$ (see e.g. \cite{BR}). An  ASM $A$ and its pattern $\widetilde{A}$ have the same term rank: $\rho(A)=\rho(\widetilde{A})$. 
As already remarked, the term rank of an $n\times n$  ASM 
may equal $n$. In this section we obtain a lower bound for the term rank of an ASM and then provide a construction to show that it is the  best possible.

\begin{theorem}\label{th:termrank}
Let $A$ be an $n\times n$ ASM. Then
\begin{equation}\label{eq:termrank}
\rho(A)\ge\left\lceil 2\sqrt{n+1}-2\right\rceil.\end{equation}
\end{theorem}

\begin{proof}
Let the term rank of $A$ be $t$. Then there exist $e$ rows and $f$ columns of $A$ with $e+f=t$ that contain all the $\pm 1$s  of $A$. We permute the rows and columns of $A$ bringing these $e$ rows to the top and bringing
these $f$ columns to the left, otherwise respecting the order of the rows and columns. The resulting matrix $A'$ has the form
\[\left[\begin{array}{cc}
A_1&A_2\\
A_3&O_{n-e,n-f}\end{array}\right]\]
where $A_1$ is an $e\times f$ matrix. The columns of $A_2$ and the rows of $A_3$ satisfy the alternating sign  property of ASMs. Thus  the sum of the entries of $A_2$ equals $n-f$ and the sum of the entries of $A_3$ equals  $n-e$.  Let $p_i$ and $n_i$ denote, respectively,  the number of $+1$s and number of $-1$s in $A_i$ for $i=1,2,3$.  Then, using the alternating sign  property, we see that the following relations hold:
\begin{equation}\label{eq:stuff}
p_2-n_2=n-f, p_3-n_3=n-e, (p_1+p_2)-(n_1+n_2)=e,
(p_2+p_3)-(n_1+n_3)=f.\end{equation}
The relations in (\ref{eq:stuff}) imply that
\begin{equation}\label{eq:morestuff}
n_1-p_1=n-(e+f)=n-t,\end{equation}
that is, $A_1$ has $n-t$ more $-1$s than $+1$s.  Since $A_1$ is an $e\times f$ matrix and since $e+f=t$, we obtain
\begin{equation}\label{eq:evenmorestuff}
n_1-p_1=n-t\le ef\le \left(\frac{t}{2}\right)^2.\end{equation}
Manipulating (\ref{eq:evenmorestuff}),  we get
\begin{eqnarray*}
n-t\le \frac{t^2}{4}\\
4n+4\le t^2+4t+4=(t+2)^2\\
2\sqrt{n+1}\le t+2\\
2\sqrt{n+1}-2\le t,\\
\end{eqnarray*}
and (\ref{eq:termrank}) follows.
\end{proof}

We now show by construction that for each $n\ge 1$ there exists an $n\times n$ ASM for which equality holds in (\ref{eq:termrank}). 
We consider the  two cases:
\begin{equation}
\label{eq:newbie1} 2k \mbox{ if } k(k+1)<n+1\le (k+1)^2\mbox{ for some integer $k$},\end{equation}
and 
\begin{equation}\label{eq:newbie2}
2k+1\mbox{ if } (k+1)^2<n+1\le (k+1)(k+2)\mbox{ for some integer $k$}.\end{equation}

First consider the case given by (\ref{eq:newbie1}), and let $B$ be any 
$k\times k$ $(0,-1)$-matrix with exactly $q=n-2k$\; $-1$s. 
Applying an elementary ASM expansion to $B$ we obtain an ASM $A$ of order $n$ all of whose nonzero entries are contained in the $k$ rows and $k$ columns of its submatrix $B$. Hence $\rho(A)\le 2k$. Since by (\ref{eq:newbie1}), $2k$ is a lower bound  for the term rank,  we have $\rho(A)=2k$ as desired.

Now consider the case given by (\ref{eq:newbie2}), and let $B$ be any $k\times (k+1)$ $(0,-1)$-matrix with $q=n-(2k+1)$\; $-1$s. 
Again, applying an elementary ASM expansion to $B$ we obtain an ASM $A$ of order $n$ all of whose nonzero entries are contained in the $k$ rows and $k+1$ columns of its submatrix $B$. Hence $\rho(A)\le 2k+1$. Since by (\ref{eq:newbie2}), $2k+1$ is a lower bound  for the term rank,  we have $\rho(A)=2k+1$ as desired.

This completes the construction that shows that equality can be obtained in Theorem \ref{th:termrank} for all $n$.
As an example of this construction, let $k=3$ and  $n=14$ so that we are  in the case given by (\ref{eq:newbie1}). Let $B$ be the $3\times 3$ $(0,-1)$-matrix with $q=n-2k=14-6=8$\; $-1$s whose only zero is in its lower right corner. Then a matrix produced by an elementary expansion is
\begin{equation}\label{eq:ex15}
\left[\begin{array}{c|c|c|c|c|c|c|c|c|c|c|c|c|c}
&&&+&&&&&&&&&&\\ \hline
&&&&&&&+&&&&&&\\ \hline
&&&&&&&&&&&+&&\\ \hline
+&&&\cellcolor[gray]{0.8} -&+&&&\cellcolor[gray]{0.8} -&+&&&\cellcolor[gray]{0.8}
-&+&\\ \hline
&&&+&&&&&&&&&&\\ \hline
&&&&&&&+&&&&&&\\ \hline
&&&&&&&&&&&+&&\\ \hline
&+&&\cellcolor[gray]{0.8} -&&+&&\cellcolor[gray]{0.8} -&&+&&\cellcolor[gray]{0.8} -&&+\\ \hline
&&&+&&&&&&&&&&\\ \hline
&&&&&&&+&&&&&&\\ \hline
&&&&&&&&&&&+&&\\ \hline
&&+&\cellcolor[gray]{0.8} -&&&+&\cellcolor[gray]{0.8} -&&&+&\cellcolor[gray]{0.8} 0&&\\ \hline
&&&+&&&&&&&&&&\\ \hline
&&&&&&&+&&&&&&\end{array}\right].\end{equation}
The nonzeros can be covered with the 3 rows and 3 columns containing the $-1$s.

\begin{corollary}\label{cor:termequal}
The minimum term rank of an $n\times n$ ASM equals
\[\left\lceil 2\sqrt{n+1}-2\right\rceil.\]
\end{corollary}

\section{Symmetric ASMs}

The $n\times n$ diamond ASM $D_n$  is a symmetric matrix. As noted in Section 1,
we can view an $n\times n$  symmetric ASM $A=[a_{ij}]$ as the  adjacency matrix
of a signed  loopy graph (an ASLG) $G(A)$ with linearly ordered vertices $1,2,\ldots,n$ with a positive
(respectively, negative) edge between vertices $i$ and $j$ if and only if  $a_{ij}=+1$ (respectively, $a_{ij}=-1$).  We denote a signed  edge $\{i,j\}$ by $\{i,j\}^+$ and $\{i,j\}^-$, respectively.  The defining property of an ASLG is:  If $i$ is any vertex and $\{j_1,j_2,\ldots,j_k\}$ are the vertices joined to $i$ by an edge ($k$ is necessarily odd) where $1\le j_1< j_2<\cdots<j_k\le n$, then  we have $\{i,j_1,\}^+, \{i,j_3\}^+,\ldots, \{i,j_k\}^+$ and $\{i,j_2\}^-,\{i,j_4\}^-,\ldots.\{i,j_{k-1}\}^-$. 
We observe that if $A$ is an  $n\times n$ ASM with only zeros on the main diagonal, then $n$ must be even. This is because the degrees of the corresponding loop-free graph $G(A)$ are all odd, and so $G(A)$ has an even number of vertices.
 Another  example of an ASLG with loops is given by the ASM
\[\left[\begin{array}{c|c|c|c|c|c|c}
&&&+&&&\\ \hline
&+&&-&+&&\\ \hline
&&+&&-&+&\\ \hline
+&-&&+&&-&+\\ \hline
&+&-&&+&&\\ \hline
&&+&-&&+&\\ \hline
&&&+&&&\end{array}\right].\]
Its signed loopy graph is a 5-cycle with vertices $2,3,4,5,6$ and  signed edges $\{2,4\}^-$, $\{4,6\}^-$, $\{6,3\}^+$,  $\{3,5\}^-$,  $\{5,2\}^+$, $\{3,4\}^-$, $\{4,1\}^+$, and $\{4,7\}^+$, and  positive loops at each of the vertices $2$, $3$, $4$, $5$, and $6$.

The elementary expansion construction used  in Section 2 to show that every $(0,+1,-1)$-matrix is a submatrix of some ASM
can be used in a symmetrical way to show that every symmetric $(0,1,-1)$-matrix  $B$ is a submatrix of a symmetric ASM $A$. If there are all $0$s on the main diagonal of $B$, the construction can be carried out so that $A$ has all $0$s on its main diagonal. 
Thus every signed (loopy) graph is an induced subgraph of an alternating signed (loopy) graph.


We now turn our attention to determining the largest number of nonzeros in an $n\times n$ symmetric ASM with  only zeros on its main diagonal, that is,
the maximum number of edges in an ASG.
Let $\sigma(A)$ equal the number of nonzero entries of $A$, and let   $\sigma^*(A)$  equal the number of nonzeros above the main diagonal of $A$. For a symmetric ASM with a zero main diagonal we have $\sigma(A)=2\sigma^*(A)$.

For $n$ a positive integer, let
\[\alpha_n=\max\{\sigma({A}): A \mbox{ an $n\times n$ symmetric ASM}\}\]
and, for $n$ a positive even integer,  let
\[\beta_n=\max\{\sigma({A}): A \mbox{ an $n\times n$ symmetric ASM with a zero main diagonal}\}.\]
Since $D_n$ is a symmetric ASM, 
\[\alpha_n=\left\{\begin{array}{cl}
\frac{n^2+1}{2}&\mbox{ if $n$ is odd}\\
\frac{n^2}{2}&\mbox{ if $n$ is even.}\end{array}\right. \]

We first suppose that $n$ is a multiple of  4.  In this case the  diamond ASM $D_n$ with $n=4k$  is a symmetric matrix  whose main diagonal consists of $k$ consecutive 0s, followed by  
$2k$ consecutive $+1$s, followed by $k$ consecutive 0s. Taking these $2k$ $+1$s and every other  pair of $-1$s on the superdiagonal and subdiagonal, we get $k$ disjoint principal $2\times 2$ submatrices of the form
\[\left[\begin{array}{cc}
+&-\\-&+\end{array}\right].\]
Replacing each of these $2\times 2$ submatrices of $D_n$ with a $2\times 2$ zero matrix we get a symmetric ASM $D_n^*$ with a zero main diagonal, and we call this ASM a {\it type $1$ hollowed-diamond ASM} of order $n=4k$.

For instance, for $k=2$ we have 
\[
D_8=\left[\begin{array}{c|c||c|c|c|c||c|c}
&&&&+&&& \\ \hline
&&&+&-&+&& \\ \hline\hline
&&+&-&+&-&+& \\ \hline
&+&-&+&-&+&-&+ \\ \hline
+&-&+&-&+&-&+& \\ \hline
&+&-&+&-&+&& \\ \hline\hline
&&+&-&+&&& \\ \hline
&&&+&&&& \end{array}\right]\rightarrow
\left[\begin{array}{c|c||c|c|c|c||c|c}
&&&&+&&& \\ \hline
&&&  +&-&+&& \\ \hline\hline
&&0&0&+&-&+& \\ \hline
&+&0&0&-&+&-&+ \\ \hline
+&-&+&-&0&0&+& \\ \hline
&+&-&+&0&0&& \\ \hline\hline
&&+&-&+&&& \\ \hline
&&&+&&&& \end{array}\right]=D_8^*.
\]
By taking the rows of the diamond ASM in the reverse order and replacing $k$ disjoint principal submatrices of the form
\begin{equation}\label{eq:anew}
\left[\begin{array}{cc}
-&+\\+&-\end{array}\right]\end{equation}
by $2\times 2$ zero matrices, we get a different ASM, which we  call a {\it type $2$ hollowed-diamond ASM} of order $4k$ and also designate as $D_n^*$. For $k=2$, we get the matrices

\[D_8=\left[\begin{array}{c|c||c|c|c|c||c|c}
&&&+&&&&\\ \hline
&&+&-&+&&&\\ \hline\hline
&+&-&+&-&+&&\\ \hline
+&-&+&-&+&-&+&\\ \hline
&+&-&+&-&+&-&+\\ \hline
&&+&-&+&-&+&\\ \hline\hline
&&&+&-&+&&\\ \hline
&&&&+&&&\end{array}\right]\rightarrow
\left[\begin{array}{c|c||c|c|c|c||c|c}
&&&+&&&&\\ \hline
&&+&-&+&&&\\ \hline\hline
&+&0&0&-&+&&\\ \hline
+&-&0&0&+&-&+&\\ \hline
&+&-&+&0&0&-&+\\ \hline
&&+&-&0&0&+&\\ \hline\hline
&&&+&-&+&&\\ \hline
&&&&+&&&\end{array}\right]=D_8^*.\]

\smallskip\noindent
Notice that the type 1 and type 2  hollowed-diamond ASMs are not equivalent under the action of the dihedral group of order 8. We refer to both of them as hollowed-diamond ASMs.

The number of nonzero entries of a  hollowed-diamond ASM of order $n=4k$ (of either type) is
\[\sigma({D_n^*})=\sigma({D_n})-4k=2(2k)^2-4k=8k^2-4k=\frac{n^2-2n}{2}.\]
Thus we have an ASG without loops of order $n=4k$  with $\frac{n^2-2n}{4}=4k^2-2k$ edges. There is another $n\times n$ symmetric ASM with zero diagonal with the same number of nonzero entries.
It is obtained from an ASM $E_n$ of order $n=4k$ with four fewer nonzero entries than $D_{4k}$  by replacing $k-1$ disjoint $2\times 2$ principal submatrices of  the form (\ref{eq:anew}) with zero matrices to get a symmetric ASM $E_n^*$ with zero diagonal. We call $E_n$ a {\it near-diamond ASM} and $E_n^*$ a {\it hollowed-near-diamond ASM}.  We illustrate this construction, again for $n=8$.
\[E_8=\left[\begin{array}{c|c||c|c|c|c||c|c}
&&&&&+&&\\ \hline
&&&&+&-&+&\\ \hline\hline
&&&+&-&+&-&+\\ \hline
&&+&-&+&-&+&\\ \hline
&+&-&+&-&+&&\\ \hline
+&-&+&-&+&&&\\ \hline\hline
&+&-&+&&&&\\ \hline
&&+&&&&&\end{array}\right]
\rightarrow
\left[\begin{array}{c|c||c|c|c|c||c|c}
&&&&&+&&\\ \hline
&&&&+&-&+&\\ \hline\hline
&&&+&-&+&-&+\\ \hline
&&+&0&0&-&+&\\ \hline
&+&-&0&0&+&&\\ \hline
+&-&+&-&+&&&\\ \hline\hline
&+&-&+&&&&\\ \hline
&&+&&&&&\end{array}\right]=E_8^*.
\]

\begin{theorem}\label{th:maxsym}
Let $n\equiv 0  \mbox{ mod $4$}$.
Then $\beta_n=\frac{n^2-2n}{2}$. Equivalently, the maximum number of edges in an ASG of order $n$ without loops is $\frac{n^2-2n}{4}$.
Moreover, a symmetric  ASM of order $n$ with zero diagonal has
$ \frac{n^2-2n}{2}$ nonzero entries if and only if it is a hollowed-diamond ASM of type $1$ or $2$, or a hollowed-near-diamond ASM.
\end{theorem}

\begin{proof} Let $n=4k$ and let $A=[a_{ij}]$ be an $n\times n$  symmetric ASM with a zero diagonal. 
We first observe that   a $2\times 2 $ submatrix of $A$ of the form \begin{equation}\label{eq:sub1}
\left[\begin{array}{cc}
a_{i-1,i}&a_{i-1,i+1}\\
a_{ii}=0&a_{i,i+1}\end{array}\right]
\mbox{ contains at least one other zero  ($i=2,3,\ldots,n-1$)}.
\end{equation}
Otherwise, by symmetry,  the submatrix of $A$ determined by rows and columns $i-1,i,i+1$ is of the form
\[\left[\begin{array}{ccc}
0&e&f\\
e&0&g\\
f&g&0\end{array}\right]\]
where $e,f,g\ne 0$.  Thus $f=-e$  and $g=-e$, and so $f=g$, violating the alternating sign property of $A$.

Consider the partition of $A$ given by
\begin{equation}\label{eq:partition}      
A=\left[\begin{array}{ccc}
A_1&X&Y\\
X^t&A_2&Z\\ 
Y^t&Z^t&A_3\end{array}\right]\end{equation}
where $A_1$ and $A_3$ are $k\times k$ symmetric matrices with zero diagonals, and $A_2$ is a $2k\times 2k$ symmetric matrix with a zero diagonal.
Then, it follows from (\ref{eq:sub1}) that
$A_2$ contains at least $k-1$ zeros above its main diagonal, that is, 
\begin{equation}\label{eq:sub2}
\sigma^*(A_2)\le 2k^2-2k+1.
\end{equation}
For the hollowed-diamond ASM $D_{4k}^*$, the submatrices corresponding to $A_1,A_3,\mbox{ and }Y$ are zero matrices,  and
the number of zeros above the main diagonal in its   $2k\times 2k$ submatrix corresponding to $A_2$ is exactly $k$. For the hollowed-near-diamond ASM $E_{4k}$, we have that $A_1$ and $A_3$ are zero matrices, the matrix $Y$ has one nonzero entry, a $+$ in its lower left corner, and  the number of zeros above the main diagonal in $A_2$ is $k-1$.

Now assume that $A$ has the maximum number of nonzeros among all symmetric ASMs of order $n$ with a zero main diagonal. 
Then since   $\sigma(D_{4k}^*)=\sigma(E_{4k}^*)=8k^2-4k$ and $\sigma^*(D_{4k}^*)=\sigma^*(E_{4k}^*)=4k^2-2k$, we have that 
\begin{equation}\label{eq:whynot}
\sigma(A)\ge 8k^2-4k \mbox{ and } \sigma^*(A)\ge 4k^2-2k.
\end{equation}
  The number of nonzeros of $A$ in the first $k$ rows (respectively, columns) is at most $1+3+\cdots+(2k-1)=k^2$, and hence, using (\ref{eq:sub2})  and (\ref{eq:whynot}), we get
\begin{eqnarray*}
4k^2-2k\le \sigma^*(A)&\le& \sigma^*(A_2)+2k^2-\sigma(Y)\\
&\le& 2k^2-2k+1
+2k^2-\sigma (Y)\\
&=&4k^2-2k+1-\sigma(Y).\end{eqnarray*}
From this calculation we see that
\begin{equation}\label{eq:sub3}
0\le \sigma(Y)\le 1,\end{equation}
and
\begin{equation}\label{eq:sub4}
\sigma \left(\left[\begin{array}{ccc}A_1&X&Y\end{array}\right]\right)\ge k^2-1
\mbox{ and }
\sigma \left(\left[\begin{array}{c}Y\\Z\\A_3\end{array}\right]\right)\ge k^2-1.
\end{equation}
Since the number of nonzeros in each row is odd, it follows from (\ref{eq:sub4}) that the numbers of nonzeros in rows $1,2,\ldots,k$ of $A$ are $1,3,\ldots,2k-1$, respectively, and, similarly,
the numbers of nonzeros in the last $k$ columns of $A$ are $2k-1,\ldots,3,1$, respectively. This implies that the first $k$ rows of $A$, after deletion of zero columns,  is a $k\times (2k-1)$ matrix of the form (shown here for $k=5$)
\begin{equation}
\label{eq:specform}
\left[\begin{array}{c|c|c|c|c|c|c|c|c}
&&&&+&&&&\\ \hline
&&&+&-&+&&&\\ \hline
&&+&-&+&-&+&&\\ \hline
&+&-&+&-&+&-&+&\\ \hline
+&-&+&-&+&-&+&-&+\end{array}\right].\end{equation}
If in $A_1$ there were a nonzero, and so a $+1$, then by symmetry there
is an $i$ with $1\le i\le k$, such that there is at least one  $+1$ below the main diagonal in row $i$  of $A_1$  and at least one $+1$ above the main diagonal in column $i$ of  $A_1$, But then, using the above formation for the first $k$ rows of $A$,   there would  be a nonzero in the  diagonal position $(i,i)$ in $A_1$, a contradiction.
Thus $A_1$, and similarly $A_3$,  is a zero matrix. 

By (\ref{eq:sub3}), we have only two cases to consider.

\smallskip\noindent
Case 1: $\sigma(Y)=0$: In this case, since $A_1,A_3,Y,Y^t$ are zero matrices and all row and column sums of $A$ equal 1, the sum of all the entries of $A$ outside of $A_2$ equals $4k$, and hence
the sum of  the entries of $A_2$ is zero.
 Since $A_2$ has a zero main diagonal,  $\sigma^*(A_2)$ is even. Thus equality cannot occur in (\ref{eq:sub2}) and hence $\sigma^*(A_2)\le 2k^2-2k$. Therefore,
\[\sigma^*(A)=\sigma^*(A_2)+2k^2\le 2k^2-2k+2k^2=4k^2-4k\]
as desired. If equality occurs, then $\sigma^*(A_2)=2k^2-2k$ and $A_2$ has exactly $2k$ zeros, with $k$ of them above the main diagonal.  With the structure already determined for $A$, it is now straightforward to verify that $A$ is a hollowed-diamond ASM of type 1 or type 2.

\smallskip\noindent
Case 2: $\sigma(Y)=1$: In this case, 
\[\sigma^*(A)=\sigma^*(A_2)+2k^2-1\le 2k^2-2k+1+2k^2 -1=4k^2-2k\]
as desired. If equality occurs, then $\sigma^*(A_2)=2k^2-2k+1$ and $A_2$ has exactly $2(2k-1)$ zeros, with $k-1$ above the main diagonal. Again, with the structure already determined for $A$, it is straightforward to check that $A$ is a hollowed-near-diamond ASM.
\end{proof}

We now assume that $n\equiv 2$  mod $4$. In this case, the diamond
ASM $D_n$ with $n=4k+2$ is a symmetric matrix whose  main diagonal consists of $k+1$ consecutive 0s, followed by $2k$ consecutive  $-1$s, followed by $k+1$ consecutive $0$s.  Taking these $2k$ consecutive $-1$s and every other pair of $+1$s on the superdiagonal and subdiagonal, we get $k$ disjoint principal $2\times 2$ submatrices of the form
\begin{equation}\label{eq:ohwell}
\left[\begin{array}{cc}
-&+\\ 
+&-\end{array}\right].\end{equation}
Replacing each these $2\times 2$ submatrices of $D_n$ with a $2\times 2$ zero matrix, we get a symmetric ASM $D_n^{\ast}$ with a zero main diagonal. 
 If we take the rows of $D_n$ in the reverse order,  the $-1$s on the main diagonal are replaced with $+1$s, and the $2\times 2$ submatrix (\ref{eq:ohwell}) is also reversed.
Continuing with the terminology used in the case $n=4k$, we call both of these matrices  {\it hollowed-diamond ASMs} of order $4k+2$. For example, with $k=2$ we have
\[D_{10}^*=\left[\begin{array}{c|c|c||c|c|c|c||c|c|c}
&&&&&+&&&&\\ \hline
&&&&+&-&+&&&\\ \hline
&&&+&-&+&-&+&&\\ \hline\hline
&&+&&&-&+&-&+&\\ \hline
&+&-&&&+&-&+&-&+\\ \hline
+&-&+&-&+&&&-&+&\\ \hline
&+&-&+&-&&&+&&\\ \hline\hline
&&+&-&+&-&+&&&\\ \hline
&&&+&-&+&&&&\\ \hline
&&&&+&&&&&\end{array}\right].\]
The number of nonzero entries of the hollowed-diamond ASM of order $n=4k+2$ is
\[\sigma(D_{4k+2}^{\ast})=\sigma(D_{4k+2})-4k=2(2k+1)^2-4k=8k^2+4k+2=\frac{n^2-2n+4}{2}.\]

\begin{theorem}\label{th:maxsym2}
Let $n\equiv 2  \mbox{ mod $4$}$.
Then $\beta_n=\frac{n^2-2n+4}{2}$. Equivalently, the maximum number of edges in an ASG of order $n$ without loops is $\frac{n^2-2n+4}{2}$. Moreover, a symmetric  ASM of order $n$ with zero diagonal has
$ \frac{n^2-2n+4}{2}$ nonzero entries if and only if it is a hollowed-diamond ASM.

\end{theorem}

\begin{proof} The proof is similar to the proof of Theorem \ref{th:maxsym} and we shall be more brief.

We now consider a symmetric ASM $=[a_{ij}]$ of order $n=4k+2$ with a zero diagonal, partitioned as in (\ref{eq:partition}), where $A_1$ and $A_3$ are now $(k+1)\times (k+1)$ matrices and, as before,  $A_2$ is $2k\times 2k$.
As in the proof of Theorem \ref{th:maxsym},
$A_2$ contains at least $k-1$ zeros above its main diagonal, and hence $\sigma^*(A_2)\le 2k^2-2k+1$.
For the hollowed-diamond ASM $D_{4k+2}^*$,   we have
$\sigma(D_{4k+2}^*) =8k^2+4k+2$ and thus  $\sigma^*(D_{4k+2}^*)=4k^2+2k +1$. The submatrix of $D_{4k+2}^*$ corresponding to $A_2$  has  a total of  $4k^2-4k$ nonzero entries, and hence  has above its main diagonal, $2k^2-2k$ nonzero entries  and $k$ zeros.

Assume that $A$ has the maximum number of nonzeros among all 
symmetric ASMs of order $n$  with a zero main diagonal. Then $\sigma^*(A)\ge\sigma^*(D_{4k+2}^*)= 4k^2+2k+1$.
Calculating, as in the proof of Theorem \ref{th:maxsym}, we get
\begin{eqnarray*}
4k^2+2k+1\le \sigma^*(A)&\le& \sigma^*(A_2)+2(k+1)^2-\sigma(Y)\\
&\le &2k^2-2k+1
+2(k+1)^2-\sigma (Y)\\
&=&4k^2+2k+3-\sigma(Y).\end{eqnarray*}
Hence
  \begin{equation}\label{eq:sub5}
0\le \sigma(Y)\le 2.\end{equation}
If $\sigma(Y)=2$, then the first $k$ rows of $A$ have the form shown
in (\ref{eq:specform}) for $k=5$, and all the nonzeros in these rows must be above the main diagonal of $A$. But then $\sigma(Y)\ge 3$, a contradiction. Thus $\sigma(Y)=0\mbox{ or }1$.

Suppose that $\sigma(Y)=0$. Since the last row of $X$ and the first column of $Z$ each contain at most  $2k$ nonzero entries, we now get that 
\begin{equation}\label{eq:similar}
4k^2+2k+1\le \sigma^*(A)=\sigma^*(A_2)+2((k+1)^2-1)\le 4k^2+2k+1.
\end{equation}
This implies that the last row of $X$ and the first column of $Y$ each contain $2k$ nonzero entries.  Thus, since each row of an ASM contains an odd number of nonzero entries,  row $k+1$ and column $3k+2$ of $A$ each contain the maximum number $2k+1$ of nonzero entries, and, for $1\le i\le k+1$,  row $i$ contains the maximum number $2i-1$
nonzero entries and column $3k+1+i$  contains the maximum number $2i-1$ nonzero entries. Thus the first $k+1$ rows of $A$, after deletion of zero columns, is a $(k+1)\times (2k+1)$ matrix of the form in (\ref{eq:specform}). This leads to a contradiction as in the proof of Theorem \ref{th:maxsym}. Thus we must have $\sigma(Y)=1$ (as in $D_{4k+2}^*$).

Now with $\sigma(Y)=1$, as before we conclude that
$A$ has $2i-1$ nonzeros in rows $i$ and columns  $4k+3-i$ for $i=1,2,\ldots,k+1$.
Moreover, the $2i-1$ nonzeros in these rows and columns must be above the main diagonal, as in $D_{4k+2}^*$.  This allows us to conclude that  $A=D_{4k+2}^*$ and $A$ has $\frac{n^2-2n+4}{2}$ nonzero entries where $n=4k+2$. The theorem now follows.
\end{proof}


If we drop the symmetry assumption, it is not difficult to determine the maximum number of nonzero entries in an ASM whose main diagonal contains only zeros.  In fact, this maximum is
\[
{n\choose 2}\mbox{ if $n\equiv 0,3$ mod 4}\]
and
\[{n\choose 2}-1\mbox{ if $n\equiv 1,2$ mod 4}.\]
We illustrate ASMs that achieve these maxima for $n=9, 10, 11, \mbox{ and }12$ from which the general pattern can be discovered.
\[
(n=9)\quad\left[\begin{array}{c|c||c|c|c|c|c||c|c}
0&&&&+&&&&\\ \hline
&0&&+&-&+&&&\\ \hline\hline
&+&0&-&+&-&+&&\\ \hline
+&-&+&0&-&+&-&+&\\ \hline
&+&-&+&0&-&+&-&+\\ \hline
&&+&-&+&0&-&+&\\ \hline
&&\cellcolor[gray]{0.8} 0&+&-&+&0&&\\ \hline\hline
&&&&+&-&+&0&\\ \hline
&&&&&+&&&0\end{array}\right] \mbox{ (35 nonzeros)}\]

\[
(n=10)\quad\left[\begin{array}{c|c||c|c|c|c|c|c||c|c}
0&&&&&+&&&&\\ \hline
&0&&&+&-&+&&&\\ \hline\hline
&&0&+&-&+&-&+&&\\ \hline
&&\cellcolor[gray]{0.8} 0&0&+&-&+&-&+&\\ \hline
&&+&-&0&+&-&+&-&+\\ \hline
&+&-&+&-&0&+&-&+&\\ \hline
+&-&+&-&+&-&0&+&&\\ \hline 
&+&-&+&-&+&\cellcolor[gray]{0.8} 0&0&\\ \hline\hline
&&+&-&+&&&&0&\\ \hline
&&&+&&&&&&0\end{array}\right]\mbox{ (44 nonzeros)}\]

\[
(n=11)\quad\left[\begin{array}{c|c|c||c|c|c|c|c||c|c|c}
0&&&&&&+&&&&\\ \hline
&0&&&&+&-&+&&&\\ \hline
&&0&&+&-&+&-&+&&\\ \hline\hline
&&+&0&-&+&-&+&-&+&\\ \hline
&+&-&+&0&-&+&-&+&-&+\\ \hline
+&-&+&-&+&0&-&+&-&+&\\ \hline
&+&-&+&-&+&0&-&+&&\\ \hline
&&+&-&+&-&+&0&&&\\ \hline\hline
&&&+&-&+&-&+&0&&\\ \hline
&&&&+&-&+&&&0&\\ \hline
&&&&&+&&&&&0\end{array}\right]\mbox{ (55 nonzeros)}\]

\[
(n=12)\quad\left[\begin{array}{c|c|c||c|c|c|c|c|c||c|c|c}
0&&&&&&+&&&&&\\ \hline
&0&&&&+&-&+&&&&\\ \hline
&&0&&+&-&+&-&+&&&\\ \hline\hline
&&+&0&-&+&-&+&-&+&&\\ \hline
&+&-&+&0&-&+&-&+&-&+&\\ \hline
+&-&+&-&+&0&-&+&-&+&-&+\\ \hline
&+&-&+&-&+&0&-&+&-&+&\\ \hline
&&+&-&+&-&+&0&-&+&&\\ \hline
&&&+&-&+&-&+&0&&&\\ \hline\hline
&&&&+&-&+&-&+&0&&\\ \hline
&&&&&+&-&+&&&0&\\ \hline
&&&&&&+&&&&&0\end{array}\right]\mbox{ (66 nonzeros)}\]

\section{Maximal ASMs}

Let $A=[a_{ij}]$ and $B=[b_{ij}]$ be $n\times n$ ASMs. Then $B$ is an {\it  ASM extension} of $A$  provided $A\ne B$ and
\[a_{ij}\ne 0\mbox{ implies } b_{ij}=a_{ij}\quad (i,j=1,2,\ldots,n).\]
In particular, if $B$ is an ASM extension of $A$, then their patterns satisfy
\[\widetilde{A}\le \widetilde{B}\; \mbox{(entrywise) and } \widetilde{A}\ne \widetilde{B}.\]
If $B$ is an ASM extension of $A$, then $E=B-A$ is a $(0,+1,-1)$-matrix with all row and column sums equal to $0$ satisfying $B=A+E$.
We call  an  ASM $A$  {\it maximal} provided it does not have an ASM extension.
The identity matrix $I_n$ and the back-identity matrix $I_n^*$ (corresponding to the permutation $(n,n-1,\ldots,2,1)$) are  maximal ASMs; this can be argued by induction using the fact that an ASM with $+1$s in opposite corners has no other nonzeros in rows and columns $1$ and $n$.  It can also be checked that 
\[\left[\begin{array}{c|c|c|c|c}
&&&+&\\ \hline
+&&&&\\ \hline
&&&&+\\ \hline
&&+&&\\ \hline
&+&&&\end{array}\right]\]
is a maximal ASM.
The permutation matrix
\[\left[\begin{array}{c|c|c|c|c}
&+&&&\\ \hline
&&&+&\\ \hline
+&&&&\\ \hline
&&&&+\\ \hline
&&+&&\end{array}\right]\]
is not a maximal ASM, since
\[\left[\begin{array}{c|c|c|c|c}
&+&&&\\ \hline
&&&+&\\ \hline
+&-&+&&\\ \hline
&+&-&&+\\ \hline
&&+&&\end{array}\right]\]
is an ASM
obtained by replacing the $2\times 2$ zero matrix in rows 3 and 4, and columns 1 and 2 with
\begin{equation}\label{eq:ee}\left[\begin{array}{cc}
-&+\\+&-\end{array}\right].\end{equation}

Generalizing this example, let $P$ be an $n\times n$ permutation matrix corresponding to the permutation $\pi$ of $\{1,2,\ldots,n\}$. 
Let $1\le p<q\le n$ and $1\le k<l\le n$.
Let $T_{p,q;k,l}=[t_{ij}]$ be the $n\times n$ $(0,+1,-1)$-matrix such that
\[t_{ij}=\left\{\begin{array}{rl}
+1&\mbox{ if  $(i,j)=(p,k)$  or $(q,l)$}\\
-1&\mbox{ if  $(i,j)=(p,l)$ or $(q,k)$}\\
0&\mbox{ otherwise.}\end{array}\right. \]
Suppose that the $2\times 2$ submatrix determined by rows $p$ and $q$ and columns $k$ and $l$ of the ASM $P$ is a zero matrix. Then $P+T_{p,q;k,l}$ is an ASM extension of $P$ provided that
\[\pi(p)>l, \pi(q)<k, \pi^{-1}(k)>q,\mbox{ and }\pi^{-1}(l)<p.\]
Pictorially, we have
\[\begin{array}{cccc|c|c|cc}
&&&&(k) &(l)&&\\
&&&&&&&\\
&&&&&\oplus&&\\ 
&&&\phantom{A}&&&\phantom{A}&\\ \hline
(p)&&&&0&0&&\oplus \\ \hline
(q)&\oplus&&&0&0&&\\ \hline
&&&&&&&\\
&&&&\oplus&&&\end{array}\]
where the $0$s are in rows $p$ and $q$ and columns $k$ and $l$, and the $\oplus$s indicate the relative positions of the $+$s of $P$ in rows $p$ and $q$ and columns $k$ and $l$. Replacing $T_{p,q;k,l}$
with $-T_{p.q;k,l}$, we get a similar picture with the permutation $\pi$ satisfying
\[\pi(p)<k, \pi(q)>l, \pi^{-1}(k)<p,\mbox{ and }\pi^{-1}(l)>q.\]
Under these circumstances, we call $P\pm T_{p,q;k,l}$ an {\it elementary ASM extension} of $P$.
 Thus if $B$ is an  elementary ASM extension of $P$, then $B$ has $n+4$ nonzero entries of which exactly two are $-1$.

We now show that if a permutation matrix $P$ has an ASM extension,
then it has an elementary ASM extension.

\begin{theorem}\label{th:elemext}
Let $P$ be an $n\times n$ permutation matrix such that $P$ is not a maximal ASM. Then $P$ has an elementary ASM extension.
\end{theorem}

\begin{proof}
Let $B$ be an ASM extension of $P$ where,
as above, we
distinguish  each  $+1$ of $P$ by $\oplus$.
Consider the $(0,+1,-1)$-matrix $E=B-P$. We construct a bipartite graph $G$, the bipartition of  whose vertices is given by the set $U$ of positions in which $E$ has $+1$s and the set $V$ of positions in which $E$ has $-1$s. There are an equal number of $+1$s and $-1$s in each row and in each column of $E$ so that $|U|=|V|$.  In each row and in each column of $B$  a  $\oplus$, if not the first or last $+$ in its row or column,  lies between two consecutive $-1$s of $B$, and thus there are an equal number of $+1$s and $-1$s of $B$ on either side (horizontally and vertically) of such a  $\oplus$; the $+1$s and $-1$s alternate from $+1$ to $-1$ to the left (respectively, above) a $\oplus$ and from $-1$ to $+1$ to the right (respectively, below) a $\oplus$. The {\it horizontal edges} of $G$ 
are obtained by  joining by an edge (i) the position of a $+1$ which is to  the left of the $\oplus$ in its row to the position containing the   $-1$ in its row that follows it, and  (ii)  the position of a $-1$ which is to the right of the $\oplus$ in its row to the position of the $+1$ that follows it.
The {\it vertical edges} of $G$ are similarly defined. Thus both the horizontal edges and the vertical edges form  perfect matchings of $G$, and each vertex of $G$ has degree equal to $2$, meeting exactly one horizontal edge and exactly one vertical edge. It follows that the edges of $G$ can be partitioned into cycles $\gamma_1,\ldots,\gamma_k$, whose edges alternate between horizontal and vertical. 
This is illustrated in Figures 2 and 3.

\bigskip\medskip
\centerline{\begin{tabular}{cc}
\begin{tabular}{|c|c|c|c|c|c|c|c|c|}
\hline
&$\oplus$&&&&&&&\\\hline
$\oplus$&$-$&&&+&&&&\\\hline
&&&+&$-$&$\oplus$&&&\\\hline
&&+&$-$&$\oplus$&$-$&+&&\\\hline
&+&$-$&$\oplus$&$-$&+&$-$&+&\\\hline
&&$\oplus$&$-$&+&$-$&+&&\\\hline
&&&+&$-$&+&&$-$&$\oplus$\\\hline
&&&&+&&$-$&$\oplus$&\\\hline
&&&&&&$\oplus$&&\\\hline
\end{tabular}
&\begin{tabular}{c}\includegraphics{asm.1}\end{tabular}
\end{tabular}}

\bigskip
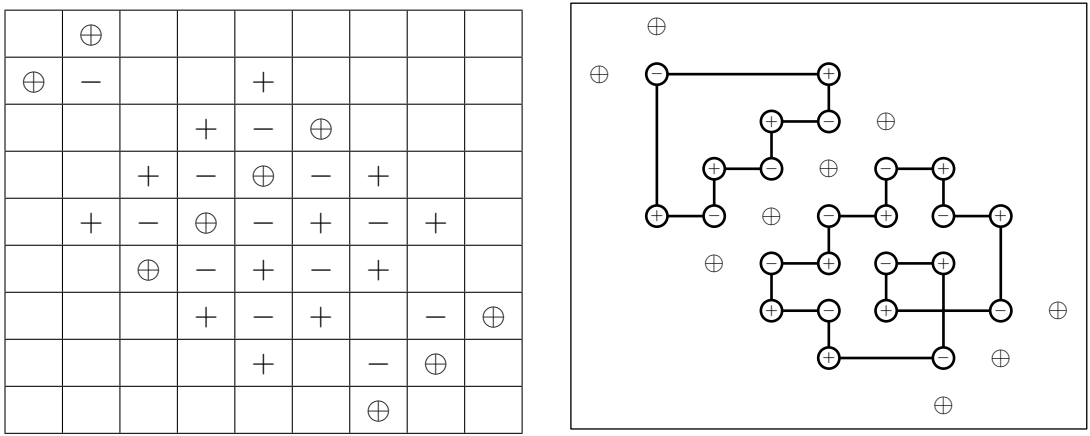
\centerline{\bf Figure 2.}
\medskip
\centerline{\begin{tabular}{cc}
\begin{tabular}{|c|c|c|c|c|c|}
\hline
&&&&$\oplus$&\\\hline
&&$\oplus$&&&\\\hline
&$\oplus$&$-$&+&&\\\hline
&&+&&$-$&$\oplus$\\\hline
$\oplus$&&&$-$&+&\\\hline
&&&$\oplus$&&\\\hline
\end{tabular}
&\begin{tabular}{c}\includegraphics{asm.2}\end{tabular}
\end{tabular}}

\bigskip
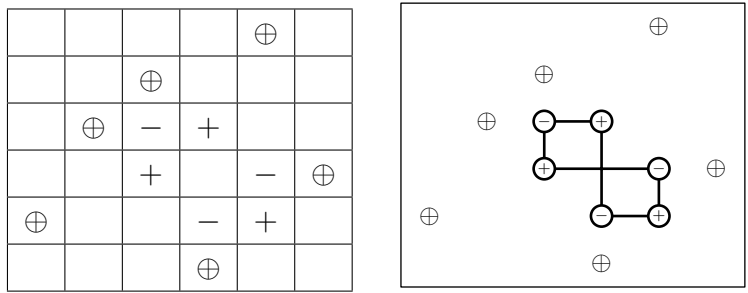
\centerline{\bf Figure 3.}

Continuing, let $C_i$ be the $(0,+1,-1)$-matrix whose nonzero entries are those
corresponding to the vertices of $\gamma_i$ $(i=1,2,\ldots,k)$. Then
it follows that $A+C_i$ is an ASM which is an extension of $A$; in fact,
$A+\sum_{i\in K}C_i$ is an extension of $A$ for all $\emptyset\ne K\subseteq \{1,2,\ldots,k\}$.  Thus we may now assume that $B=A+C_1$, that is, $G$ is a cycle of even length.

Consider the top row of $C_1$ that contains a position $(p,q)$ with a $-1$ (so there is a $\oplus$ above it) and assume without loss of generality that the edge of $\gamma_1$ goes to a $+1$ to its right, We follow $\gamma_1$ until we arrive at the first position $(r,s)$ with a $-1$ for which the edge of $\gamma_1$ goes to the left; such a position exists since $\gamma_1$ is a cycle. There are two possibilities according to whether we arrive at this position $(r,s)$ from above it or below it. See Figure 2.

If we arrive from above, then initial $-1$ in position $(p,q)$ along with the  $-1$ in the position $(r,s)$ give the matrix
$-T_{p,r:q,s}$. Considering  the positions containing a $\oplus$, we see that  $A+T_{p,r;q,s}$ is an elementary extension of $A$. Suppose we arrive at position $(r,s)$ from below it. Let  $(g,h)$ be the  position containing a $-1$  that came before $(r,s)$. Then  considering again where the $\oplus$s are, we see that $A-T_{g,r;h,s}$ is an elementary extension of $A$. See Figure 3.
\end{proof}

\begin{corollary}\label{cor:elemext}
Let $P$ be an $n\times n$ permutation matrix  corresponding to the permutation
$\pi$ of $\{1,2,\ldots,n\}$. Then $P$ is a maximal ASM if and only if
there do not exist integers $p,q,k,l$ with $1\le p<q\le n$ and $1\le k<l\le n$ such that 
\[\pi(p)>l, \pi(q)<k, \pi^{-1}(k)>q,\mbox{ and }\pi^{-1}(l)<p.\]
or
\[\pi(p)<k, \pi(q)>l, \pi^{-1}(k)<p,\mbox{ and }\pi^{-1}(l)>q.\]
\end{corollary}

Elementary extensions are obtained  by appropriately adding $\pm T_{p,q;k,l}$ to an ASM $A$ in such a way that the nonzero positions of $A$ and those of $T_{p,q:k,l}$ do not overlap.
It is also possible to add $\pm T_{p,q;k,l}$ to an ASM  where the nonzero positions overlap and the result is  an ASM. For example,
\[\left[\begin{array}{rrr}
+1&0&0\\
0&+1&0\\
0&0&+1\end{array}\right]
+
\left[\begin{array}{rrr}
-1&+1&0\\
+1&-1&0\\
0&0&0\end{array}\right]
=
\left[\begin{array}{rrr}
0&+1&0\\
+1&0&0\\
0&0&+1\end{array}\right].
\]
In fact, with two overlapping positions and starting with an ASM equal to  the permutation matrix $P$ as in the example, the result is another permutation matrix $Q$ if and only if the permutation corresponding to $Q$ is obtained from the permutation corresponding to $P$  by a transposition. Another  example, this time  with an overlap of one position, is
\[\left[\begin{array}{rrrrr}
0&0&+1&0&0\\
0&+1&0&0&0\\
+1&0&-1&0&+1\\
0&0&0&+1&0\\
0&0&+1&0&0\end{array}\right]+
\left[\begin{array}{rrrrr}
0&0&0&0&0\\
0&0&-1&+1&0\\
0&0&+1&-1&0\\
0&0&0&0&0\\
0&0&0&0&0\end{array}\right]=
\left[\begin{array}{rrrrr}
0&0&+1&0&0\\
0&+1&-1&+1&0\\
+1&0&0&-1&+1\\
0&0&0&+1&0\\
0&0&+1&0&0\end{array}\right]\]

The above examples lead us to to define an {\it ASM-interchange} to be the operation (or result thereof) of adding $\pm T_{p,q:r,s}$ to an ASM provided the result is also an ASM.
As shown above, ASM-interchanges generalize transpositions
of permutations.

We now show that every ASM can be generated from the special ASM
$I_n$ by a sequence of ASM-interchanges. Since we can generate all permutation matrices from $I_n$ by transpositions, this is equivalent to showing that every ASM results from the special subclass of
ASMs consisting of permutation matrices by a sequence of ASM-interchanges. In fact, we can get every ASM from a permutation matrix by ASM-interchanges with overlap of only 1 or 2.

\begin{theorem} Let $A$ be an $n\times n$ ASM. Then there is a sequence of ASM-interchanges such that starting from the identity matrix $I_n$ we obtain $A$.
\end{theorem}

\begin{proof}
If $A$ does not have any $-1$s, then $A$ is a permutation matrix
and we are done. Suppose $A=[a_{ij}]$ has at least one $-1$, and choose the  first row that has  a $-1$ and the first column in that row that has a $-1$. Suppose this $-1$ is in position $(q,l)$, that is, $a_{ql}=-1$. Since $A$ is an ASM, there exists unique $p<q$ such that 
$a_{pl}=+1$ and $a_{il}=0$ for all $i\ne p$ with $i<q$. Similarly, there exists unique  $k<l$ such that $a_{qk}=+1$ and $a_{qj}=0$ for all
$j\ne q$ with $j<l$. Then $A+T_{p,q;k,l}$ is an ASM with one less  $-1$ than $A$. Proceeding inductively, there is a sequence of ASM-interchanges  which reduces $A$ to an ASM  $B$ without any $-1$s.
Thus $B$ is a permutation matrix, and further ASM-interchanges (i.e. transpositions) reduces $A$ to $I_n$.
\end{proof}

\section{Concluding Remarks}

We conclude with some comments and open questions. As we have observed, every $(0,+1,-1)$-matrix (possibly rectangular)  is a submatrix of some ASM, and  every symmetric $(0,1,-1)$-matrix is a principal submatrix of some  symmetric ASM.  If $B$ is a $(0,+1,-1)$-matrix, we define   $\zeta(B)$ to be the smallest $n$ such that $B$ is a submatrix of an $n\times n$ ASM, and  
\[\zeta(k,l)=\max\{\zeta(B): B\mbox{ a  $k\times l$ $(0,1,-1)$-matrix}\}\quad (k,l\ge 1).\]
Note that if $B$ is a submatrix of an $n\times n$ ASM, then by taking a direct sum with an identity matrix, $B$ is a submatrix of an $m\times m$ ASM for all $m\ge n$.
  
Similarly, if  $C$  is a symmetric $(0,+1,-1)$-matrix, we define  $\zeta_p(C)$ to be the smallest $n$ such that $C$ is a principal submatrix of an $n\times n$ symmetric ASM, and 
\[\zeta_p(k)=\max\{\zeta_p(C): C\mbox{ a  $k\times k$ symmetric $(0,1,-1)$-matrix}\}\quad (k\ge 1).\] 

Let
\[B=\left[\begin{array}{c|c}
-&-\\ \hline
&+\end{array}\right].\]
Then 
\[A=\left[\begin{array}{c|c|c|c|c}
&&&+&\\ \hline
&+&&&\\ \hline
+&-&+&-&+\\ \hline
&&&+&\\ \hline
&+&&&\end{array}\right]\]
is an ASM whose submatrix determined by rows 3 and 4 and columns 2 and 4 equals $B$. Although $B$ is not symmetric, we can insert a 
new row and column to get an ASM $A'$ which contains $B$ as a principal submatrix in rows and columns 3 and 5:
\[\left[\begin{array}{c|c|c|c|c|c}
&&&&+&\\ \hline
&&+&&&\\ \hline
&+&-&+&-&+\\ \hline
+&&&&&\\ \hline
&&&&+&\\ \hline
&&+&&&\end{array}\right].\]

\smallskip\noindent
{\bf Question 1:}  What are the values of $\zeta(k,l)$ and
$\zeta_p(k)$, and which $(0,+1,-1)$-matrices achieve these maxima?

\smallskip
Let $J_{k} $ be the $k\times k$ matrix all of whose entries equal $+1$.
Then $J_k$ is a (principal) submatrix of the diamond ASM $D_{4k-3}$ and $-J_k$ is a (principal) submatrix of the diamond ASM $D_{4k-1}$.
It is natural to {\it conjecture} that $\zeta(k,k)=4k-1$, equivalently, that every
$k\times k$ $(0,+1,-1)$-matrix is a submatrix of  a $(4k-1)\times (4k-1)$ ASM.
\smallskip

In Section 6 we have seen that a permutation matrix may or may not be maximal, that is, may or may not have an ASM extension.

\smallskip\noindent
{\bf Question 2:}  Characterize maximal ASMs.

\smallskip

 Let $A$ be an $n\times n$ ASM. As previously remarked, 
  the 7 non-identity row and column permutations of $A$ corresponding to the dihedral group of order $8$ result in ASMs, some of  which may equal $A$.
  Let ${\mathcal C}(A)$ denote the equivalence class of all ASMs that can be gotten from $A$ by permuting rows and columns. Thus if $A$ is a permutation matrix, then ${\mathcal C}(A)$ is the set of all permutation matrices and has cardinality $n!$.

\smallskip\noindent
{\bf Question 3:}  What is 
\begin{equation}\label{eq:form0}
\max\{|{\mathcal C}(A)|: \mbox{ $A$ an  $n\times n$  ASM}\}?\end{equation}

\smallskip
From the above,  this maximum is at least $n!$ but, in general, it is much larger. 
 For example, consider 
the $n\times n$ ASM $A$ which is the direct sum of  $l$ copies of $D_3$ and one $I_1$ (so $n=3l+1$).
The number of ASMs in ${\mathcal C}(A)$ is readily seen to be
 \begin{equation}\label{eq:form1}
 \frac{1}{l!}{n\choose 3}^2{{n-3}\choose 3}^2\cdots {4\choose 3}^2=\frac{n!^2}{l!6^{2l}}.\end{equation}
 If $l\ge 2$, then (\ref{eq:form1}) is greater than $n!$.
 If  $A'$ is the direct sum of  $l-1$ copies of  $D_3$ and one  $D_4$, the number of ASMs in ${\mathcal C}(A')$ is
\begin{equation}\label{eq:form2}
2\frac{1}{(l-1)!}{n\choose 3}^2{{n-3}\choose 3}^2\cdots {7\choose 3}^2=\frac{2n!^2}{(l-1)!6^{2l}16}.\end{equation}
(The factor of 2 is due to the fact that the rows of $D_4$ can be reversed to create the other matrix we have referred to as $D_4$.)
Subtracting (\ref{eq:form1}) from (\ref{eq:form2}), we see that if $l>8$, 
\[|{\mathcal C}(A')|>|{\mathcal C}(A)|.\]
These calculations illustrate that the determination of (\ref{eq:form0}) is apt to be very difficult. A related question is:

\smallskip\noindent
{\bf Question 4:}  What is 
\begin{equation}\label{eq:form0}
\max\{|{\mathcal C}(A)|: \mbox{ $A$ an  $n\times n$ connected ASM}\}?\end{equation}

\end{document}